\documentclass[letterpaper, 10pt, conference]{ieeeconf}  

\IEEEoverridecommandlockouts                              
\usepackage{mathtools} 

\usepackage[sort,compress]{cite}
\usepackage{amsfonts,dsfont,amssymb,bm}

\usepackage[amsmath,thmmarks]{ntheorem}
\usepackage{subcaption}

\usepackage{microtype}

\usepackage{algorithm}
\usepackage{algpseudocode}
\usepackage{xcolor}

\DeclarePairedDelimiter{\abs}{\lvert}{\rvert}

\usepackage{array,graphicx,verbatim,mathrsfs,bbm,mathtools}

\DeclareMathOperator{\ee}{\mathbb{E}}			
\DeclareMathOperator{\prob}{{\mathds{P}}}			

\usepackage{algorithm}
\usepackage{algpseudocode}

\makeatletter
\newcommand*\doTRANS[2]{\raisebox{\depth}{$\m@th#1\intercal$}}
\makeatother

\usepackage{xparse}
\NewDocumentCommand\AVG{s}
    {\IfBooleanTF#1%
      {\frac{1}{\abs{N  }} \sum_{i \in N  }}
      {\frac{1}{\abs{N^m}} \sum_{i \in N^m}}
    }

\usepackage{xcolor}

\usepackage [english]{babel}
\usepackage [autostyle, english = american]{csquotes}
\MakeOuterQuote{"}

\title{Optimal Symmetric Strategies in Multi-Agent Systems with Decentralized Information}
\author{Sagar Sudhakara and  Ashutosh Nayyar
\thanks{S. Sudhakara and A. Nayyar are with the Department of Electrical \& Computer
Engineering, University of Southern California, Los Angeles, CA 90089
(e-mail: sagarsud@usc.edu; ashutosn@usc.edu). }}

\begin{document}
\maketitle
\thispagestyle{empty}
\pagestyle{empty}

\begin{abstract}
    
We consider a cooperative multi-agent system consisting of a  team of agents with decentralized information.
Our focus is on the design of \emph{symmetric} (i.e. identical) strategies for the agents in order  to optimize a finite horizon team objective. 
We start with a general information structure and then consider some special cases. The constraint of using symmetric strategies introduces new features and complications in the team problem. For example, we show  in a simple example that randomized symmetric strategies may outperform deterministic symmetric strategies. We also discuss why some of the known approaches for reducing agents’ private information in teams
may not work under the constraint of symmetric
strategies. We then adopt the common information approach  for our
problem and modify it to accommodate the use of symmetric
strategies. This results in a common information based dynamic
program where each step involves minimization over a single
function from the space of an agent’s private information
to the space of probability distributions over actions. We present specialized models where private
information can be reduced using simple dynamic program
based arguments.


\end{abstract}

\section{Introduction}\label{sec;Introduction}
The problem of sequential decision-making by a team of collaborative agents operating in an uncertain environment has received significant attention in the recent control (e.g. \cite{nayyar2010optimal, nayyar2013decentralized, mahajan2013optimal,sudhakara2021optimal,kartik2022optimal}) and artificial intelligence (e.g. \cite{seuken2008formal,kumar2015probabilistic,rashid2018qmix,hu2019simplified, szer2005maa}) literature. The goal in such problems is to design decision/control strategies for the multiple agents in order to optimize a performance metric for the team.


In some cooperative multi-agent (or team) problems, the agents are essentially identical and interchangeable. For example, consider a team of autonomous agents operating in an environment. The agents may have identical sensors that they use to observe their local surroundings   and they may have identical action spaces. For teams with such identical agents,  it may be convenient for the designer to design identical decision/control strategies for the agents. This would be particularly helpful if the number of agents is large --instead of designing $n$ different strategies for $n$ agents in a team, the designer needs to design just one strategy for all agents. Identical strategies may also be necessary for other practical and regulatory reasons. For example, a self-driving car company would be expected to have the same control algorithm on all its cars. Another reason for using identical strategies arises in situations where agents don't  have any individualized identities. This can happen in  settings where the population of the agents is not fixed and agents are unaware of the total number of agents currently present or  their own index in the population. An example of such a situation for a multi-access communication problem is described in  \cite{neely2021repeated}. When an agent doesn't know its own index ("Am I agent 1 or agent 2?"), it makes sense to use symmetric (i.e. identical) strategies  for all agents irrespective of their index.   In this paper, we will focus on the design of identical strategies for a team of cooperative agents. We will refer to such strategies as \emph{symmetric strategies}. 



Our focus  is on designing symmetric strategies to optimize a finite horizon team objective. We start with a general information structure and then consider some special cases. The constraint of using symmetric strategies introduces new features and complications in the team problem. For example, when agents in a team are free to  use individualized strategies,   it is well-known that agents can be restricted to deterministic strategies without loss of optimality  \cite{yuksel2013stochastic}. However, we show in a simple example that randomized strategies may be helpful when the agents are constrained to use symmetric strategies.


We adopt the common information approach \cite{nayyar2013decentralized} for our problem and modify it  to accommodate the use of symmetric strategies. This results in a common information based dynamic program where each step involves minimization over a single function from the space of an agent's private information to the space of probability distributions over actions.  The complexity of this dynamic program depends in large part on the size of the private information space. We discuss some known approaches for reducing agents' private information  and why they may \emph{not} work under the constraint of symmetric strategies. We present two specialized models where private information can be reduced using simple dynamic program based arguments.


\textit{Notation:}
Random variables are denoted by upper case letters (e.g. $X$), their realization with lower case letters (e.g. $x$), and their space of realizations by script letters (e.g. $\mathcal{X}$). Subscripts denote time and superscripts denote agent index; e.g., $X^i_t$ denotes the state of agent $i$ at time $t$. The short hand notation $X^i_{1:t}$ denotes the collection $(X^i_1,X^i_2,...,X^i_t)$.  $\triangle(\mathcal{X})$ denotes the probability simplex for the space $\mathcal{X}$. $\prob(A)$ denotes the probability of an event $A$. $\ee[X]$ denotes the expectation of a random variable $X$. $\mathds{1}_A$ denotes the indicator function of event $A$. For simplicity of notation, we use $\prob(x_{1:t},u_{1:t-1})$ to denote $\prob(X_{1:t}=x_{1:t},U_{1:t-1}=u_{1:t-1})$ and a similar notation for conditional probability.
For a strategy pair $(g^1,g^2)$, we use $\prob^{(g^1,g^2)}(\cdot)$ (resp. $\ee^{(g^1,g^2)}[\cdot]$) to indicate that the probability (resp. expectation) depends on the choice of the strategy pair.
We use $-i$ to denote agent/agents other than agent $i$. $U \sim \lambda$ indicates that $U$ is randomly distributed according to the distribution $\lambda$.


\section{Problem Formulation}\label{sec:problem_formulation}

Consider a discrete-time system with two agents. The system state consists of three components - a shared state and two local states, one for each agent.  $X^i_t \in \mathcal{X}$ denotes the local state of agent $i$, $i=1,2$, at time $t$ and $X^0_t \in \mathcal{X}^0$ denotes the shared state at time $t$. $X_t$ denotes the triplet $(X^0_t, X^1_t,X^2_t)$.     Let  $U^i_t \in \mathcal{U}$ denote the control action of agent $i$ at time $t$. $U_t$ denotes the pair $(U^1_t,U^2_t)$.
 The dynamics of the shared and local states are as follows:
 \begin{equation}\label{eq:dyna_global}
    X^0_{t+1} =f^0_t(X^0_t, U_t, W^0_t),
 \end{equation}
\begin{equation}\label{eq:dyna}
    X^i_{t+1}=f_t(X^{i}_{t},X^0_t, U_t,W^i_t),~~ i=1,2, 
\end{equation}
where $W^0_t \in \mathcal{W}^0$ and $W^i_t \in \mathcal{W}$ are random disturbances with $W^0_t$ having the probability distribution $p^0_W$ and $W^i_t, i=1,2,$ having the probability distribution $p_W$. We use $W_t$ to denote the triplet $(W^0_t, W^1_t, W^2_t)$. Note that the next local state of agent $i$ depends on its own current local state, the shared state and the control actions of both the agents. Also note that the function $f_t$ in \eqref{eq:dyna} is the same for both agents. The initial  states $X^0_1,X^1_1, X^2_1$ are independent random variables with $X^0_1$ having the probability distribution $\alpha^0$ and $X^i_1, i=1,2,$ having the probability distribution $\alpha$. The initial states $X^0_1, X^1_1,X^2_1$ and the disturbances  $W^0_t, W^i_t, t \ge 1$, $i=1,2$, are independent discrete random variables. These will be referred to as the primitive random variables of the system.

\subsection{Information structure and strategies}
The information available to agent $i$ , $i=1,2,$ at time $t$ consists of two parts:
\begin{enumerate}
    \item Common  information $C_t$ - This information is available to both agents\footnote{$C_t$ does not have to be the \emph{entirety} of information that is available to both agents; it simply cannot include anything that is not available to both agents.}. $C_t$ takes values in the set $\mathcal{C}_t$.
    \item Private information $P^i_t$ - Any information available to agent $i$ at time $t$ that is not included in $C_t$ is  included in $P^i_t$.
    $P^i_t$ takes values in $\mathcal{P}_t$. (Note that the space of private information is the same for both agents.) We use $P_t$ to denote the pair $(P^1_t, P^2_t)$.
\end{enumerate}
$C_t$ should be viewed as an ordered list (or row vector) of some of the system variables that are known to both agents. Similarly, $P^i_t$ should be viewed as an ordered list (or row vector).

We assume that $C_t$ is non-decreasing with time, i.e., any variable included in $C_t$ is also included in $C_{t+1}$. Let $Z_{t+1}$ be  the increment in common information from time $t$ to $t+1$. We assume the following dynamics for $Z_{t+1}$ and $P^i_{t+1}$ ($i=1,2$):
\begin{align*}
    Z_{t+1} = \zeta_t(X_t, P_t, U_t, W_{t}); P^i_{t+1} = \xi^i_t(X_t, P_t, U_t, W_{t}), \label{eq:zeq} 
\end{align*}

Agent $i$ uses its information at time $t$ to select a probability distribution $\delta{U}^i_t$ on the action space $\mathcal{U}$. We will refer to $\delta{U}^i_t$ as agent $i$'s \emph{behavioral action} at time $t$.  The action $U^i_t$ is then randomly generated according to the chosen distribution, i.e., $U^i_t \sim \delta{U}^i_t $. Thus, we can write
\begin{equation}\label{strategy_f_1}
    \delta{U}^i_t = g^i_t(P^{i}_{t}, C_t),
\end{equation}
where    $g^i_t$ is a mapping from $\mathcal{P}_t \times \mathcal{C}_t$  to $\Delta(\mathcal{U})$.
The function $g^i_t$ is referred to as the control strategy of agent $i$ at time $t$.
  The collection of functions $g^i:= (g^i_1, \ldots, g^i_T)$ is referred to as the control strategy of agent $i$. Let $\mathcal{G}$ denote the set of all possible strategies for agent $i$. (Note that the set  of all possible strategies is the same for the two agents since  the private information space, the common information space and the action space are the same for the two agents.)

  We use $(g^1,g^2)$ to denote the pair of strategies being used by agent 1 and agent 2 respectively. We are interested in the finite horizon total expected cost incurred by the system which is defined as:
  \begin{equation}\label{eq:cost2}
  J(g^1,g^2):=\ee^{(g^1,g^2)}\left[\sum_{t=1}^{T}k_t( {X}_t,{U}_t)\right],
  \end{equation}
 where $k_t$ is the cost function at time $t$. Our focus will be on the case of \emph{symmetric strategies}, i.e., the case where both agents use the same control strategy. When referring to symmetric strategies, we will drop the superscript $i$ in $g^i$  and denote a symmetric strategy pair by $(g,g)$.

\textit{Symmetric strategy optimization problem (\textbf{Problem P1}):} 
Our objective is to find a symmetric strategy pair  that achieves the minimum total expected cost among all symmetric strategy pairs. That is, we are looking for a strategy $g \in \mathcal{G}$ such that 
\begin{equation}\label{eq:optimalg}
    J(g,g) \leq J(h,h),~ ~\forall h \in \mathcal{G}.
\end{equation}

\begin{remark}

We assume that the randomization at each agent is done independently over time and independently of the other agent \cite{kartik2021upper}.

\end{remark}


\begin{remark}
If the  private information space, the common information space and the action space are finite, then it can be shown that the strategy space $\mathcal{G}$ is a compact space and that $J(g,g)$ is a continuous function of $g \in \mathcal{G}$. Thus, an optimal $g$ satisfying \eqref{eq:optimalg} exists.
\end{remark}


\begin{remark}
We have formulated the problem with two agents for simplicity. The number of agents can in fact be any positive integer $n$ or even a deterministic time-varying sequence $n_t$. Our results extend to these cases with only notational modifications.
\end{remark}

\begin{remark}
    Note that we are not claiming that use of symmetric strategies is always optimal -- it is not. We are simply focusing on the design of symmetric strategies for reasons mentioned in the introduction.
\end{remark}

\subsection{Some specific information structures}\label{sec:specific_IS}
We will be particularly interested in the  special cases of Problem P1 described below. Each case corresponds to a different information structure.
In each case, the shared state history until time $t$, $X^0_{1:t}$, and the action history until $t-1$, $U_{1:t-1}$ are part of common information $C_t$. 

1. \emph{One-step delayed sharing information structure:} In this case, each agent knows its own local state history until time $t$ and the local state history of the other agent until time $t-1$. 
Thus, the common and private information available to agent $i$ at time $t$  is given by 
\begin{align}
   &C_{t}=(X^0_{1:t},U_{1:t-1}, X^{1,2}_{1:t-1}); ~P^i_t = X^i_{t}.
\end{align}
We refer to the instance of Problem P1 with this information structure as \textbf{Problem P1a}.

2. \emph{Full local history information structure:}  In this case, each agent knows its own local state history until time $t$ but does not observe the local states of the other agent. Thus, the common and private information available to agent $i$ at time $t$  is given by 
\begin{align}
   &C_{t}=(X^0_{1:t},U_{1:t-1}); ~P^i_t = X^i_{1:t}.
\end{align}
This information structure corresponds to the  control sharing information structure of \cite{mahajan2013optimal}. We refer to the instance of Problem P1 with this information structure as \textbf{Problem P1b}.

3. \emph{Reduced local history information structure:}  In this case, each agent knows its own \emph{current} local state but does not recall its past local states and does not observe the local states of the other agent. Thus, the common and private information available to agent $i$ at time $t$  is given by 
\begin{align}
   &C_{t}=(X^0_{1:t},U_{1:t-1}); ~P^i_t = X^i_{t}.
\end{align}
We refer to the instance of Problem P1 with this information structure as \textbf{Problem P1c}.


    Another special case of Problem P1 that might be of interest is the following: Consider a situation where the state dynamics are governed not by the vector of agents' actions but only by an aggregate effect of agents' actions. Let $A_t =a(U^1_t,U^2_t)$ denote the aggregate action. We refer to $a(\cdot,\cdot)$ as the aggregation function. Some examples of $a$ could be the sum or the maximum function. The state dynamics are as described in equations \eqref{eq:dyna_global} and \eqref{eq:dyna} except with $U_t$ replaced by $A_t$. The agents only observe the aggregate actions taken in the past but not the individual actions. The common and private information are given as: 
\begin{align}
   &C_{t}=(X^0_{1:t},A_{1:t-1}); ~P^i_t = X^i_{t}.
\end{align}
We addressed this case in Appendix \ref{App:at},

\subsection{Why are randomized strategies needed?}

In team problems, it is well-known that one can restrict agents to deterministic strategies without loss of optimality  \cite{yuksel2013stochastic}. However, since the agents are restricted to use symmetric strategies in our setup, randomization can help. This can be illustrated by the following simple example.

\emph{Example 1:}
Let $T=1$ and let $(X^0_1, X^1_1, X^2_1) =(0,0,0)$ with probability $1$. The action space is $\mathcal{U} = \{0,1\}$. The information structure is that of Problem P1c described in \ref{sec:specific_IS}.  The cost at $t=1$ is given by, $k_1({X}_1,{U}_1) = \mathds{1}_{\{U^1_{1} = U^2_{1} \}}$.

Note that the cost function penalizes the agents for taking the same action.
In this case, each agent has only two \emph{deterministic} strategies -- taking action $0$ or taking action $1$ at time $1$. If both agents use the same deterministic strategy, then, clearly, $U^1_1 = U^2_1$ and hence the expected cost incurred is $1$.

Consider now the following randomized strategy for each agent: $U^i_1 =1$ with probability $p$ and $U^i_1=0$ with probability $(1-p)$.
When the two agents use this randomized strategy,   the  expected cost is
$p^2+(1-p)^2$. 
With $p=0.5$, this cost is $0.5$ which is less than the expected cost achieved by any deterministic symmetric strategy pair. Thus, when agents are restricted to use the same strategy, they can benefit from randomization.

\section{Common information approach}\label{sec:CI_approach}

We adopt the common information approach \cite{nayyar2013decentralized} for Problem P1. This approach formulates a new decision-making problem from the perspective of a coordinator that knows the common information. At each time, the coordinator selects prescriptions that map each agent's private information to its action. The behavioral action of each agent in this problem is simply the prescription evaluated at the current realization of its private information. Since Problem P1 requires symmetric strategies for the two agents, we will require the coordinator to select \emph{identical prescriptions for the two agents}.  To make things precise, let $\mathcal{B}_t$ denote the space of all functions from $\mathcal{P}_t$ to $\Delta(\mathcal{U})$. Let $\Gamma_t \in \mathcal{B}_t$ denote the  prescription selected by the coordinator at time $t$.  Then, the behavioral action of agent $i$, $i=1,2,$ is given by: $\delta U^i_t = \Gamma_t(P^i_t)$.

As in Problem P1, agent $i$'s action $U^i_t$  is generated according to the distribution $\delta U^i_t$ using independent randomization.
The coordinator selects its prescription at time $t$ based on the common information at time $t$ and the history of past prescriptions. Thus, we can write: 
\begin{equation}
    \Gamma_t = d_t(C_t, \Gamma_{1:t-1}),
\end{equation}
where $d_t$ is a mapping from $\mathcal{C}_t \times \mathcal{B}_1 \ldots \times \mathcal{B}_{t-1}$ to $\mathcal{B}_t$.  The collection of  mappings $d:= (d_1,\ldots,d_T)$ is referred to as the coordination strategy. The coordinator's objective is to choose a coordination strategy that minimizes the finite horizon total expected cost:
\begin{equation}
    \mathcal{J}(d) := \ee^{d}\left[\sum_{t=1}^{T}k_t({X}_t,{U}_t)\right].
\end{equation}
The following lemma establishes the equivalence of the coordinator problem formulated above and the problem Problem P1.  The use of identical prescriptions by the coordinator is needed to connect the coordinator's strategy to symmetric strategies for the agents in Problem P1.
\begin{lemma}
Problem P1 and the coordinator's problem are equivalent in the following sense:\\
(i) For any symmetric strategy pair $(g,g)$, consider the following coordination strategy: 
\[ d_t(C_t) = g_t(\cdot, C_t).\]
Then, $J(g,g) = \mathcal{J}(d)$.
(ii) Conversely, for any coordination strategy $d$, consider the symmetric strategy pair defined as follows:
\[ g_t(\cdot, C_t) = d_t(C_t, \Gamma_{1:t-1} ),\]
where $\Gamma_k = d_k(C_k, \Gamma_{1:k-1})$ for $k=1,\ldots, t-1$.
\end{lemma}
\begin{proof}
The proof is based on Proposition 3 of \cite{nayyar2013decentralized} and the fact that the use of identical prescriptions for the two agents by the coordinator corresponds to the use of symmetric strategies in Problem P1.
\end{proof}

We now proceed with finding a solution for the coordinator's problem. As shown in \cite{nayyar2013decentralized}, the coordinator's belief on $(X_t, P_t)$ can serve as its information state  (sufficient statistic) for selecting prescriptions. At time $t$, the coordinator's belief is given as:
\begin{align}\label{coord:prob1_belief}
    &\Pi_t(x,p)=\prob(X_t=x, P_t =p|C_t, \Gamma_{1:(t-1)}),
\end{align}
for all $x \in \mathcal{X}^0 \times \mathcal{X} \times \mathcal{X}, p \in \mathcal{P}_t \times \mathcal{P}_t$. The belief can be sequentially updated by the coordinator as described in  Lemma \ref{lemma:belief_update} below. The lemma follows from arguments similar to those in Lemma 2 of \cite{kartik2021upper} (or Theorem 1 of \cite{nayyar2013decentralized}).

\begin{lemma} \label{lemma:belief_update}
For any coordination strategy $d$, the coordinator's belief $\Pi_t$ evolves almost surely as
\begin{equation}
    \Pi_{t+1} = \eta_t(\Pi_t, \Gamma_t, Z_{t+1}),
\end{equation}
where $\eta_t$ is a fixed transformation that does not depend on the coordination strategy.
\end{lemma}

Using  the results in \cite{nayyar2013decentralized}, we can write a dynamic program for the coordinator's  problem. Recall that  $\mathcal{B}_t$ is the space of all functions from $\mathcal{P}_t$ to $\Delta(\mathcal{U})$. For a $\gamma \in \mathcal{B}_t$ and $p \in \mathcal{P}_t$, $\gamma(p)$ is a probability distribution on $\mathcal{U}$. Let $\gamma(p;u)$ denote the probability assigned to $u \in \mathcal{U}$ under the probability distribution $\gamma(p)$. 

\begin{theorem}\label{thm:dp}
The value functions  for the coordinator's dynamic program are as follows: Define $V_{T+1}(\pi_{T+1}) =0$ for every $\pi_{T+1}$. For $t \leq T$ and  for any realization $\pi_t$ of $\Pi_t$, define
\begin{align}
    V_t(\pi_t) = \min_{\gamma_t \in \mathcal{B}_t} \mathbb{E}[&k_t(X_t, U_t) + \nonumber\\ 
        &V_{t+1}(\eta_t(\pi_t, \gamma_t, Z_{t+1})) | \Pi_t = \pi_t, \Gamma_t = \gamma_t]\label{eq:DP1_new}
\end{align}

 The coordinator's optimal strategy is to pick the minimizing prescription  for each time and each $\pi_t$.
\end{theorem}
\begin{proof}
As noted in \cite{nayyar2013decentralized}, the coordinator's problem can be seen as a POMDP. The theorem  is simply the POMDP dynamic program for the coordinator. 
\end{proof}

\begin{remark}
The expectation in \eqref{eq:DP1_new} should be interpreted as follows: $Z_{t+1}$ is given by \eqref{eq:zeq}, $U^i_t, i=1,2,$ is independently randomly generated according to the distribution $\gamma_t(P^i_t)$ and the joint distribution on $(X_t,P_t)$ is $\pi_t$.
\end{remark}

\begin{remark}
It can be established by backward induction that the term being minimized in \eqref{eq:DP1_new} is a continuous function of $\gamma_t$. This can be shown using an argument very similar to the one used in the proof of Lemma 3 in \cite{kartik2021common}. This continuity property along with the fact that   $\mathcal{B}_t$ is a compact set ensures that the minimum in \eqref{eq:DP1_new} is achieved.
\end{remark}

 For the instances of Problem P1 described in Problems P1a - P1c (see Section \ref{sec:problem_formulation}), the private information of an agent includes its current local state. Consequently, for these instances, the coordinator's belief is just on the private information of the agents and the current shared state.  The  following lemma shows that this belief can be factorized into beliefs on each agent's private information and a degenerate belief on the shared state.

\begin{lemma}\label{lem:lemma3}
In Problems 1a - 1c, for any realization $x^0$ of the shared state and any realizations $p^1, p^2$ of the agents' private information, 
\begin{equation}\label{eq:lemma3_eq1}
    \Pi_t(x^0,p^1,p^2) =  \delta_{X^0_t}(x^0)\Pi^1_t(p^1)\Pi^2_t(p^2),
\end{equation}
where $\Pi_t$ is the coordinator's belief (see \eqref{coord:prob1_belief}), $\Pi^1_t, \Pi^2_t$ are the marginals of $\Pi_t$ for each agent's private information and $\delta_{X^0_t}(\cdot)$ is a delta distribution located at $X^0_t$. (Recall that $X^0_t$ is part of the common information in Problems P1a-P1c.)

Further, for any coordination strategy $d$,  $\Pi^i_t, i=1,2,$ evolves almost surely as
\begin{equation}\label{eq:lemma3_eq2}
    \Pi^i_{t+1} = \eta^i_t(X^0_t, \Pi^i_t, \Gamma_t, Z_{t+1}),
\end{equation}
where $\eta^i_t$ is a fixed transformation that does not depend on the coordination strategy.
\end{lemma}
\begin{proof}
See Appendix I.
\end{proof}
Because of the above lemma, we can replace $\Pi_t$ (and its realizations $\pi_t$) by  $( \Pi^1_t, \Pi^2_t,X^0_t)$ (and the corresponding realizations $( \pi^1_t, \pi^2_t, x^0_t)$) in the dynamic program of Theorem \ref{thm:dp} for Problems P1a -P1c.





\section{Comparison of Problems 1b and 1c}\label{sec:compare_problems}

The information structures in Problems P1b and P1c differ only in the private information available to the agents -- in P1b, each agent know its entire local state history whereas in P1c each agent knows only its current local state. \emph{If the agents were not restricted to use the same strategies,} it is known that the two information structures are equivalent. That is, if a (possibly asymmetric) strategy pair $(g^1,g^2)$ is optimal for the information structure in Problem P1c, then it is also optimal for the information structure  in Problem P1b \cite{mahajan2013optimal}. This effectively means that  agents can ignore their past local states without any loss in performance. However, such an equivalence of the two information structures may not hold when agents are restricted to use symmetric strategies. In other words, an optimal symmetric strategy in Problem P1c may not be optimal for Problem P1b; and the optimal performance in Problem P1c  may be strictly worse than the optimal performance in Problem P1b.  We explore this point in more detail below.

One approach for establishing that agents can ignore parts of their private information that has been commonly used in prior literature on multi-agent/decentralized systems is the agent-by-agent (or person-by-person approach)  \cite{ho1980team,nayyar2013decentralized}.
This approach works as follows: We start by fixing strategies of all agents other than agent $i$ to arbitrary choices and then show that agent $i$ can make decisions based on a subset or a function of its private information without compromising performance. If this reduction in agent $i$'s information holds for any arbitrary strategy of other agents, we can conclude that this reduction would hold for globally optimal strategies as well.  By repeating this  argument for all agents, one can reduce the private information of all agents without losing performance. The problem with this approach is that it cannot accommodate the restriction to symmetric strategies. The  reduced-information based strategies  obtained using this approach  may or may not be symmetric. Thus, we cannot adopt this approach for reducing agents' private information in Problem P1b.

 Another  approach for reducing private information that has been used in some game-theoretic settings \cite{kartik2021common} involves the use of conditional probabilities of actions given reduced information. To see how this approach can be used,  let's consider  an arbitrary (possible asymmetric) strategy pair $(g^1, g^2)$ for the information structure of Problem P1b and define the following conditional probabilities for $i=1,2$: 
\begin{align}\label{eq:cond_prob1}
  \prob^{({g}^1,{g}^2)}[U^i_t=u|&X^i_t=x, C_t=c_t].
\end{align}
Note that \eqref{eq:cond_prob1} specifies a probability distribution on $\mathcal{U}$ for each $x$ and $c_t$. Thus, it can be viewed as a valid strategy for agent $i$ \emph{under the information structure of Problem P1c}. This observation lets us define the following reduced-information strategies for the agents:
\begin{align}
    \bar{g}^i_t(x,c_t) := \prob^{({g}^1,{g}^2)}[U^i_t=\cdot|&X^i_t=x, C_t=c_t], ~i=1,2. \label{eq:gbar}
\end{align}
Further, it can be shown that the above construction  ensures that the joint distributions  of $(X_t,U_t,C_t)$ under strategies $(g^1,g^2)$ and $(\bar{g}^1,\bar{g}^2)$ are the same for all $t$. This, in turn, implies that $J(\bar{g}^1,\bar{g}^2) = J(g^1,g^2)$.
This argument establishes that there is a reduced-information strategy pair with the same performance as an arbitrary full-information strategy pair. Thus, the optimal performance with reduced-information strategies must be the same as the optimal performance with full-information strategies for the information structure of Problem P1b.

 We can try to use the above argument for symmetric strategy pairs. We start with an arbitrary symmetric strategy pair $(g,g)$ in Problem P1b and use \eqref{eq:gbar} to define a reduced-information strategy pair that achieves the same performance as $(g,g)$. The problem with this argument is that even though we started with a symmetric strategy pair $(g,g)$, the reduced-information strategy pair constructed by \eqref{eq:gbar} need not be symmetric. Hence, this reduced-information strategy pair may not be  a valid solution for Problem P1c. We illustrate this point in the following example.

\emph{Example 2:}  Consider a setting where there is no shared state, the action space is $\mathcal{U} = \{a, b\}$ and  the local states are i.i.d. (across time and across agents). Each local state is a  Bernoulli (1/2) random variable. Consider the  symmetric strategy pair $(g,g)$ for Problem P1b where $g_1$ (the strategy at $t=1$)  is:
\begin{equation}
g_1(u^i_1=a|x^i_1)=
 \begin{cases}
    \alpha, ~ \text{if}~  x^i_1=0\\
    \beta, ~ \text{if}~  x^i_1=1,
  \end{cases} \label{strategy2}
\end{equation}
where $0 \leq \alpha, \beta \leq 1$.
And $g_2$ (the strategy at  $t=2$)  is:
\begin{equation}
g_2(u^i_2=a|x^i_1,x^i_2,u^1_1,u^2_1)=
 \begin{cases}
    \alpha, ~ \text{if} ~ x^i_1=x^i_2\\
    \beta, ~ \text{if} ~ x^i_1\neq x^i_2.
  \end{cases} \label{strategy2}
\end{equation}
We now use \eqref{eq:gbar} to define a reduced-information strategy. Even though we started with a symmetric strategy pair for the two agents, the conditional probability on the right hand side of \eqref{eq:gbar} may  be different for the two agents. To see this, consider $t=2$ and   $C_2 = (U^1_1, U^2_1)=(a,b)$ and $X^i_2=0$. Then, for agent 1: 
\begin{align}
    &\prob^{(g,g)}(U^1_2=a|X^1_2=0,U^1_1=a,U^2_1=b) \notag \\
    &=\prob^{(g,g)}(U^1_2=a,X^1_1=0|X^1_2=0,U^1_1=a,U^2_1=b)\notag\\
    &+\prob^{(g,g)}(U^1_2=a,X^1_1=1|X^1_2=0,U^1_1=a,U^2_1=b)\notag\\
    &=\alpha \prob^{(g,g)}(X^1_1=0|X^1_2=0,U^1_1=a,U^2_1=b) \notag\\
    &+ \beta \prob^{(g,g)}(X^1_1=1|X^1_2=0,U^1_1=a,U^2_1=b)\notag\\
    &=\alpha \left(\frac{\alpha}{\alpha+\beta}\right) + \beta \left(\frac{\beta}{\alpha+\beta}\right)=\frac{\alpha^2+\beta^2}{\alpha+\beta} 
    \label{eq:expresssion1}
\end{align}
On  the other hand, a similar calculation  for agent 2 shows that: $\prob^{(g,g)}(U^2_2=a|X^2_2=0,U^1_1=a,U^2_1=b)$
\begin{align}
    &=\alpha \left(\frac{1-\alpha}{2-\alpha-\beta}\right) + \beta \left(\frac{1-\beta}{2-\alpha-\beta}\right)=\frac{\alpha+\beta-\alpha^2-\beta^2}{2-\alpha-\beta}.\label{eq:expresssion2}
\end{align}
The expressions in \eqref{eq:expresssion1} and \eqref{eq:expresssion2} are clearly different. For example, with  $\alpha=1/4$ and $\beta=1/2$, \eqref{eq:expresssion1} evaluates to $5/12$ while \eqref{eq:expresssion2} evaluates to $7/20$.
Thus, the reduced-information strategies constructed by \eqref{eq:gbar} are not symmetric and, therefore, invalid for Problem P1c.

\subsection{Special cases}
In this section, we present two special cases under which Problems P1b and P1c can be shown to be equivalent, i.e., we can show that an optimal strategy for Problem P1c is also optimal for Problem P1b.

\subsubsection{Specialized cost}
We assume that the cost function at each time $t$ is non-negative, i.e.,  $k_t(X^0_t,X^1_t,X^2_t,U^1_t,U^2_t) \geq 0$.
Further, we assume that for each possible local state $x^i$ of agent $i$ there exists an action $m(x^i)$ such that   $k_t(x^0,x^1,x^2,m(x^1),m(x^2))=0$ for all $x^0 \in \mathcal{X}^0$. An example of such a cost function is  $k_t(X^0_t,X^1_t,X^2_t,U^1_t,U^2_t) = (X^1_t - U^1_t)^2[(X^2_t  -U^2_t)^2 + 1 ] + (X^2_t  -U^2_t)^2,$ where the states and actions are integer-valued.

Recall that in Problem P1b the prescription space at time $t$ is the space of functions from $\mathcal{X}^t$ to $\Delta(\mathcal{U})$ and in Problem P1c the prescription space is the space of functions from $\mathcal{X}$ to $\Delta(\mathcal{U})$. Using the dynamic programs for Problems P1b and P1c with the specialized cost above, we can show that  optimal prescriptions in both problems effectively coincide with the mapping $m$ from $\mathcal{X}$ to $\mathcal{U}$\footnote{With a slight abuse of notation, the function $m$ from $\mathcal{X}$ to $\mathcal{U}$ can be viewed as a deterministic prescription from $\mathcal{X}$ to $\Delta(\mathcal{U})$ or from $\mathcal{X}^t$ to $\Delta(\mathcal{U})$. }.

\begin{lemma}\label{thm:dpcost}
The value functions  for the coordinator's dynamic programs in Problems P1b and P1c can be written as follows:  For $t \leq T$ and  for any realization $\pi^1_{t},\pi^2_{t},x^0_{t}$ of $\Pi^1_{t},\Pi^2_{t},X^0_{t}$,
\begin{align}
    V_{t}(\pi^1_{t}&,\pi^2_{t},x^0_{t}) :=\min_{\gamma_{t} \in \mathcal{B}_t}Q_{{t}}(\pi^1_{t},\pi^2_{t},x^0_{t},\gamma_{t}), \label{eq:qfun_cost}
 \end{align}
 where the function $Q$ satisfies
  \begin{align}
      &Q_{{t}}(\pi^1_{t},\pi^2_{t},x^0_{t},\gamma_{t})\ge Q_{{t}}(\pi^1_{t},\pi^2_{t},x^0_{t},m)=0,
\end{align}
Consequently, the coordinator's optimal prescription is $m$ at each time.
\end{lemma}
\begin{proof}
See Appendix \ref{App:proof_special_cost}.
\end{proof}
Since the coordinator's optimal strategy is identical in  Problems P1b and P1c, it follows that  the optimal symmetric  strategy for the agents in the two problems is also the same, namely  $U^i_t=m(X^i_t)$. 

\subsubsection{Specialized dynamics}\label{Special Dynamics}
We consider a specialized dynamics where the local states $X^i_{1:T}, i=1,2,$  are $\mathrm{i.i.d.}$ uncontrolled random variables with probability distribution $\alpha$ and there is no shared state.
The following theorem shows the equivalence between Problems P1b and P1c in terms of optimal performance and strategies.

\begin{lemma}\label{theorem:sp_dyn}
The optimal performance in Problem P1c is the same as the optimal performance in Problem P1b. 
Further, the optimal symmetric strategy for Problem P1c is optimal for Problem P1b as well.
\end{lemma}
\begin{proof}
See Appendix \ref{App:sp_dyn}.
\end{proof}
In summary, for the specialized cases described above,   one can reduce the private information of the agents without losing performance,  even with the restriction to symmetric strategies.  


\section{Comparison of Problems 1a and 1c\label{sec:compare_problems}}

The information structures in Problems P1a and P1c differ only in the common information available to the agents -- in P1a, each agent has an additional part in the common information consisting of the local state history of both agents. Since agents in Problem P1c have less information that their counterparts in Problem P1a, an optimal symmetric strategy in Problem P1c may not optimal for Problem P1a; and the optimal performance in Problem P1c may be strictly worse then the optimal performance in Problem P1a. 

\emph{Example 3:}  Consider a setting where there is no shared state, the action space is $\mathcal{U} = \{0, 1\}$ and state space is $\mathcal{X} = \{0, 1\}$. Let $T=2$ and the local states of each agent are stationary across time. The initial states $X^1_1, X^2_1$ are independent random variables with probability distribution Bernoulli (1/2). The cost at time $t=1$ is given by $k_1(X_1,U_1)=10\mathds{1}_{\{U^1_1 \neq 0,U^2_1 \neq 0\}}$ and cost at time $t=2$ is given by:
\begin{equation}
k_2(X_2,U_2)=
 \begin{cases}
    0, ~ \text{if}~  U^1_2=X^2_2~and~ U^2_2=X^1_2\\
    1, ~ \text{otherwise},
  \end{cases} \label{strategy2}
\end{equation}

Consider the  symmetric strategy pair $(g,g)$ for Problem P1a  where $g_1$ (the strategy at $t=1$)  is: $U^1_1=0, U^2_1=0$. At time $t=2$, each agent uses the following strategy: $U^1_2=X^1_1, 
 U^2_2=X^2_1$ if $X^1_1=X^2_1$ and $U^1_2=1-X^1_1, 
 U^2_2=1-X^2_1$ if $X^1_1\neq X^2_1$. This results in optimal expected cost of $0$ in Problem P1a.
In Problem P1c, it can be shown that the optimal strategy at time $t=1$ is $U^1_1=0, U^2_1=0$ and at time $t=2$, $U^i_2$ follows probability distribution Bernoulli (1/2).
The optimal expected cost is $0.75$ for the Problem P1c, which is strictly worse than optimal performance in Problem P1a. 

\subsection{Special Case}
We present a special dynamics under which problems P1a and P1c can be shown to be equivalent i.e. we can show that an optimal strategy for Problem P1a is also optimal for Problem P1c.
The dynamics of the shared and local states in the specialized dynamics problem are as follows:
 \begin{equation}\label{eq:dyna_global2}
    X^0_{t+1} =f^0_t(X^0_t, U_t, W^0_t),
 \end{equation}
\begin{equation}\label{eq:dyna2}
    X^i_{t+1}=f_t(X^0_t, U_t,W^i_t),~~ i=1,2. 
\end{equation}
The shared and local state dynamics are similar to \eqref{eq:dyna_global} and \eqref{eq:dyna} except that in local dynamics the next local state doesn't depend on the current local state. In this case, we have the following result:

\begin{lemma} \label{lemma:1a-1c}
   For the specialized dynamics described in \eqref{eq:dyna_global2} - \eqref{eq:dyna2}, an optimal symmetric strategy in Problem P1c is also optimal for Problem P1a and, consequently, the optimal performance in the two problems are the same.
\end{lemma}
\begin{proof}
    See Appendix \ref{App:ac}.
\end{proof}

\section{Conclusion}\label{sec:conclusion}
In this paper, we focused on designing symmetric strategies
to optimize a finite horizon team objective. We started with a
general information structure and then considered some special
cases. We showed in a simple example that randomized
symmetric strategies may outperform deterministic symmetric
strategies. We also discussed why some of the known approaches
for reducing agents’ private information in teams may not work
under the constraint of symmetric strategies. We modified the common information approach to obtain 
optimal symmetric strategies for the agents.  This resulted in a common information based
dynamic program whose complexity  depends
in large part on the size of the private information space.
 We presented two specialized
models where private information can be reduced using simple
dynamic program based arguments. 





\bibliographystyle{ieeetr}
\bibliography{IEEEabrv,ref_learning}

\appendices

\section{Proof of Lemma \ref{lem:lemma3}}\label{App:Lemma_3}
To prove Lemma \ref{lem:lemma3}, we first show that the private information of the agents  are
conditionally independent given the common information under any strategies. For Problem P1a, this is straightforward since the disturbances in the dynamics are independent:
\begin{align*}\label{eq:indepenP1a}
    \prob(X^{1,2}_t&=x^{1,2}_t|C_t=(x^{1,2}_{1:t-1},x^0_{1:t},u_{1:t-1}))\notag\\
    =&\prob(f_{t-1}(x^1_{t-1},x^0_{t-1},u_{t-1},W^1_{t-1})=x^1_t)\times\notag\\
    &\prob(f_{t-1}(x^2_{t-1},x^0_{t-1},u_{t-1},W^2_{t-1})=x^2_t).
\end{align*}

For Problems P1b and P1c, we have the following lemma. 
\begin{lemma}[Conditional independence property] \label{LEM:INDEPEN}
Consider any arbitrary (symmetric or asymmetric) choice of agents' strategies in Problems P1b and P1c. Then, at any time $t$, the two agents' private information are conditionally independent given the common information $C_t$. That is, if $c_t$ is the realization of the common information at time $t$ then for any realization $p_{t}$ of private information, we have
\begin{equation}\label{eq:indepen}
    \prob^{(g^1,g^2)}({p}_{t}|c_t)=\displaystyle\prod_{i=1}^{2} \prob^{g^i}(p^{i}_{t}|c_t), 
    \end{equation}
    Further, $\prob^{g^i}(p^{i}_{t}|c_t)$ depends only on agent $i$' strategy and not on the strategy of agent $-i$.
\end{lemma}    
\begin{proof}
The proof is analogous to the proof of \cite[Proposition 1]{mahajan2013optimal} except for the possible randomization in agents' strategies.
\end{proof}
 Using the  above conditional independence property for Problems P1a-P1c, we can now prove \eqref{eq:lemma3_eq1}.
At time $t$, the coordinator's belief is given as:
\begin{equation}
    \Pi_t(x^0_t,p^1_t,p^2_t) =  \prob(X^0_t=x^0_t,P^1_t=p^1_t,P^2_t=p^2_t|C_t,\Gamma_{1:t-1})
\end{equation}
for any realization $x^0_t$ of the global state and any realizations $p^1_t, p^2_t$ of the agents' private information.
Since $X^0_t$ is part of $C_t$, the coordinator's belief can be factorized into:
\begin{align*}
     &\prob(x^0_t,p^1_t,p^2_t|C_t,\Gamma_{1:t-1})=\delta_{X^0_t}(x^0_t)\prob(p^1_t,p^2_t|C_t,\Gamma_{1:t-1})\notag\\
     &= \delta_{X^0_t}(x^0_t)\prob(p^1_t|C_t,\Gamma_{1:t-1})\prob(p^2_t|C_t,\Gamma_{1:t-1})\notag\\
     &=\delta_{X^0_t}(x^0_t)\Pi^1_t(p^1_t)\Pi^2_t(p^2_t),
\end{align*}
where we used the above-mentioned conditional independence.

We now prove \eqref{eq:lemma3_eq2} for Problems P1a-P1c.
\subsubsection{Problem P1a}
In Problem P1a, let $\pi^i_{t+1}$ be the realization of the coordinator's marginal belief $\Pi^i_{t+1}$ for each agent's private information and $c_{t+1}$ be the realization of common information at time $t+1$.
The belief for Problem P1a is given by:
\begin{align}
    \pi^i_{t+1}(x^i_{t+1})=\prob(x^i_{t+1}|c_{t+1}=(x^{1,2}_{1:t},x^0_{1:t+1},u_{1:t}),\gamma_{1:t})
\end{align}
Using Bayes' rule, we have
\begin{align*}
    &\pi^i_{t+1}(x^i_{t+1})=\frac{\prob(x^i_{t+1},x^0_{t+1}|x^{1,2}_{1:t},x^0_{1:t},u_{1:t},\gamma_{1:t})}{\prob(x^0_{t+1}|x^{1,2}_{1:t},x^0_{1:t},u_{1:t},\gamma_{1:t})}\notag\\
    &=\frac{\prob(f_t(x^0_{t},u_t,W^0_t)=x^0_{t+1})\prob(f_t(x^i_t,x^0_{t},u_t,W^i_t)=x^i_{t+1})}{\prob(f_t(x^0_{t},u_t,W^0_t)=x^0_{t+1})}\notag\\
    &=\prob(f_t(x^i_t,x^0_{t},u_t,W^i_t)=x^i_{t+1})
\end{align*}
Thus $\pi^i_{t+1}$ is determined by $x^0_{t}$ and the increment in common information.

\subsubsection{Problem P1b}
In Problem P1b, let $\pi^i_{t+1}$ be the realization of the coordinator's marginal belief $\Pi^i_{t+1}$ for each agent's private information and $c_{t+1}$ be the realization of common information at time $t+1$.
The belief for agent $1$ is given by:
\begin{align}
    \pi^1_{t+1}(x^1_{1:t+1})=\prob(x^1_{1:t+1}|c_{t+1}=(x^0_{1:t+1},u_{1:t}),\gamma_{1:t})
\end{align}
Using Bayes' rule, we have $\pi^1_{t+1}(x^1_{1:t+1})$
\begin{align}
    &=\frac{\sum_{x^{2}_{1:t}}\prob(x^1_{1:t+1},x^0_{t+1},u_t,x^{2}_{1:t}|x^0_{1:t},u_{1:t-1},\gamma_{1:t})}{\sum_{\tilde{x}^1_{1:t+1}}\sum_{\tilde{x}^{2}_{1:t}}\prob(\tilde{x}^1_{1:t+1},x^0_{t+1},u_t,\tilde{x}^{2}_{1:t}|x^0_{1:t},u_{1:t-1},\gamma_{1:t})} \label{bel_eq11}
\end{align}
The numerator of \eqref{bel_eq11} using state dynamics, coordinator prescription and belief at time $t$, can be written as
\begin{align}
    & \prob(x^1_{t+1}|x^1_t,x^0_t,u_t)\prob(x^0_{t+1}|x^0_t,u_t)\mathds{1}_{u^1_t=\gamma(x^1_{1:t})}\notag\\
    &\times \pi^1_t(x^1_{1:t})\sum_{x^{2}_{1:t}}\mathds{1}_{u^2_t=\gamma(x^2_{1:t})}\pi^2_t(x^2_{1:t}) \label{bel_eq22}
\end{align}
Similarly the denominator can be written as
\begin{align}
 \sum_{\tilde{x}^1_{1:t+1}}\sum_{\tilde{x}^{2}_{1:t}}&\prob(\tilde{x}^1_{t+1}|\tilde{x}^1_t,x^0_t,u_t)\prob(x^0_{t+1}|x^0_t,u_t)\mathds{1}_{u^1_t=\gamma(\tilde{x}^1_{1:t})}\notag\\
    &\times \mathds{1}_{u^2_t=\gamma(\tilde{x}^2_{1:t})}\pi^1_t(\tilde{x}^1_{1:t})\pi^2_t(\tilde{x}^2_{1:t}) \label{bel_eq33}
\end{align}
Let $z^b_{t+1}:=(x^0_{t+1},u_t)$ be the increment in the common information in Problem P1b. Substituting equations \eqref{bel_eq22}, \eqref{bel_eq33} in equation \eqref{bel_eq11}, we derive $\pi^1_{t+1}(x^1_{1:t+1})$ as,
\begin{align}
    &\frac{\prob(x^1_{t+1}|x^1_t,x^0_t,u_t)\prob(x^0_{t+1}|x^0_t,u_t)\mathds{1}_{u^1_t=\gamma(x^1_{1:t})} \pi^1_t(x^1_{1:t})}{\sum_{\tilde{x}^1_{1:t+1}}\prob(\tilde{x}^1_{t+1}|\tilde{x}^1_t,x^0_t,u_t)\prob(x^0_{t+1}|x^0_t,u_t)\mathds{1}_{u^1_t=\gamma(\tilde{x}^1_{1:t})}\pi^1_t(\tilde{x}^1_{1:t})}
\end{align}
We denote the update rule described above with $\eta^i_t$, i.e.
\begin{equation}
\eta^i_t(x^0_t, \pi^1_t, \gamma_t, z^b_{t+1})
\end{equation}

\subsubsection{Problem P1c}
In problem P1c, let $\pi^i_{t+1}$ be the realization of the coordinator's marginal belief $\Pi^i_{t+1}$ for each agent's private information and $c_{t+1}$ be the realization of common information at time $t+1$.
The belief for agent $1$ is given by:
\begin{align}
    \pi^1_{t+1}(x^1_{t+1})=\prob(x^1_{t+1}|c_{t+1}=(x^0_{1:t+1},u_{1:t}),\gamma_{1:t})
\end{align}
Using Bayes' rule, we have $\pi^1_{t+1}(x^1_{t+1})$
\begin{align}
    &=\frac{\sum_{x^{1,2}_t}\prob(x^1_{t+1},x^0_{t+1},u_t,x^{1,2}_t|x^0_{1:t},u_{1:t-1},\gamma_{1:t})}{\sum_{\tilde{x}^1_{t+1},\tilde{x}^{1,2}_t}\prob(\tilde{x}^1_{t+1},x^0_{t+1},u_t,\tilde{x}^{1,2}_t|x^0_{1:t},u_{1:t-1},\gamma_{1:t})} \label{bel_eq1}
\end{align}
The numerator of \eqref{bel_eq1} using state dynamics, coordinator prescription and belief at time $t$, can be written as
\begin{align}
    & \sum_{x^{1,2}_t}\prob(x^1_{t+1}|x^1_t,x^0_t,u_t)\prob(x^0_{t+1}|x^0_t,u_t)\mathds{1}_{u^1_t=\gamma(x^1_t)}\notag\\
    &\times \mathds{1}_{u^2_t=\gamma(x^2_t)}\pi^1_t(x^1_t)\pi^2_t(x^2_t) \label{bel_eq2}
\end{align}
Similarly the denominator can be written as
\begin{align}
 \sum_{\tilde{x}^1_{t+1}}\sum_{\tilde{x}^{1,2}_t}&\prob(\tilde{x}^1_{t+1}|\tilde{x}^1_t,x^0_t,u_t)\prob(x^0_{t+1}|x^0_t,u_t)\mathds{1}_{u^1_t=\gamma(\tilde{x}^1_t)}\notag\\
    &\times \mathds{1}_{u^2_t=\gamma(\tilde{x}^2_t)}\pi^1_t(\tilde{x}^1_t)\pi^2_t(\tilde{x}^2_t) \label{bel_eq3}
\end{align}
Let $z^c_{t+1}:=(x^0_{t+1},u_t)$ be the increment in the common information in problem P1c. Substituting equations \eqref{bel_eq2}, \eqref{bel_eq3} in equation \eqref{bel_eq1}, we derive $\pi^1_{t+1}(x^1_{t+1})$ as
\begin{align}
    &\frac{\sum_{x^{1}_t}\prob(x^1_{t+1}|x^1_t,x^0_t,u_t)\prob(x^0_{t+1}|x^0_t,u_t)\mathds{1}_{u^1_t=\gamma(x^1_t)} \pi^1_t(x^1_t)}{\sum_{\tilde{x}^1_{t+1},\tilde{x}^{1}_t}\prob(\tilde{x}^1_{t+1}|\tilde{x}^1_t,x^0_t,u_t)\prob(x^0_{t+1}|x^0_t,u_t)\mathds{1}_{u^1_t=\gamma(\tilde{x}^1_t)}\pi^1_t(\tilde{x}^1_t)}
    \end{align}
We denote the update rule described above with $\eta^i_t$, i.e.
\begin{equation}
\eta^i_t(x^0_t, \pi^1_t, \gamma_t, z^c_{t+1})
\end{equation}

\label{App:Lemma_3}
\section{Proof of Lemma \ref{LEM:UPDATEspecdyn}}\label{App:lemma_special_dyn}
Then the coordinator's belief state can serve as the sufficient statistic for selecting prescriptions. Let $c_{t+1}:=(x_{1:t},x^0_{1:t+1},u_{1:t})$ be the realization of common information $C_{t+1}$ and $\gamma_{1:t}$ be the realization of the prescription $\Gamma_{1:t}$. The coordinator's belief in Problem P1a for all $x^i_{t+1} \in \mathcal{X}$ is given as:
\begin{align}\label{coord:prob1a}
    &\pi^i_{t+1}(x^{i}_{t+1})=\prob(X^i_{t+1}=x^{i}_{t+1}|x^{1,2}_{1:t},x^0_{1:t+1},u_{1:t}, \gamma_{1:t}),\notag\\
    &=\frac{\prob(X^i_{t+1}=x^i_{t+1},X^0_{t+1}=x^0_{t+1}|x^{1,2}_{1:t},x^0_{1:t},u_{1:t}, \gamma_{1:t})}{\prob(X^0_{t+1}=x^0_{t+1}|x^{1,2}_{1:t},x^0_{1:t},u_{1:t}, \gamma_{1:t})}\notag\\
    &=\frac{\prob(f^0_t(x^0_{t}, u_{t}, W^0_{t})=x^0_{t+1})\prob(f_t(x^0_{t}, u_{t},W^i_{t})=x^i_{t+1})}{\prob(f^0_t(x^0_{t}, u_{t}, W^0_{t})=x^0_{t+1})}\notag\\
    &=\prob(f_t(x^0_{t}, u_{t},W^i_{t})=x^i_{t+1}).
\end{align}

Let $c_{t+1}:=(x^0_{1:t+1},u_{1:t})$ be the realization of common information $C_{t+1}$ and $\gamma_{1:t}$ be the realization of the prescription $\Gamma_{1:t}$. The coordinator belief in Problem P1c for all $x^i_{t+1} \in \mathcal{X}$ is given as:
\begin{align}\label{coord:prob1c}
    &\Pi^i_{t+1}(x^{i}_{t+1})=\prob(X^i_{t+1}=x^{i}_{t+1}|x^0_{1:t+1},u_{1:t}, \gamma_{1:t}),\notag\\
    &=\frac{\prob(X^i_{t+1}=x^i_{t+1},X^0_{t+1}=x^0_{t+1}|x^0_{1:t},u_{1:t}, \gamma_{1:t})}{\prob(X^0_{t+1}=x^0_{t+1}|x^0_{1:t},u_{1:t}, \gamma_{1:t})}\notag\\
    &=\frac{\prob(f^0_t(x^0_{t}, u_{t}, W^0_{t})=x^0_{t+1})\prob(f_t(x^0_{t}, u_{t},W^i_{t})=x^i_{t+1})}{\prob(f^0_t(x^0_{t}, u_{t}, W^0_{t})=x^0_{t+1})}\notag\\
    &=\prob(f_t(x^0_{t}, u_{t},W^i_{t})=x^i_{t+1}).
\end{align}\label{App:lemma_special_dyn}

\section{Proof of Lemma \ref{thm:dpcost}}\label{App:proof_special_cost}
We prove the lemma by backward induction. Let's consider Problem P1b.  The value function for the coordinator's dynamic program at time $T$ can be written as follows: for any realization $\pi^1_T,\pi^2_T,x^0_T$ of $\Pi^1_T,\Pi^2_T,X^0_T$ respectively, 
 \begin{align}
  V_{T}(\pi^1_T,\pi^2_T,& x^0_T) =\min_{\gamma_T \in\mathcal{B}_T}Q_{T}(\pi^1_T,\pi^2_T,x^0_T,\gamma_T),
  \end{align}
  where
  \begin{align}
  Q_{T}(\pi^1_T,\pi^2_T,& x^0_T,\gamma_T):=\sum_{x_{1:T}}\sum_{u_T} k_T(x_T,u_T)\delta_{x^0_T}(x^0)\pi^1_T(x^{1}_{1:T}) \notag\\
  &\times \pi^2_T(x^{2}_{1:T}) \gamma_T(x^{1}_{1:T};u^1_T)\gamma_T(x^{2}_{1:T};u^2_T).
 \end{align}
$Q_{T}(\pi^1_T,\pi^2_T,x^0_T,\gamma_T)\ge 0$ because $k_T(\cdot,\cdot)$ is a non negative function.
The deterministic mapping $m$ from $\mathcal{X}$ to $\mathcal{U}$ can be viewed as a prescription $\gamma \in \mathcal{{B}_{T}}$ with $\gamma(x^{i}_{1:T};m(x^i_T))=1$.
\begin{align}
    Q_{T}(\pi^1_T,\pi^2_T,& x^0_T,m):=\sum_{x_{1:T}}\sum_{u_T} k_T(x_T,m(x^1_T),m(x^2_T))\notag\\
    &\times \pi^1_T(x^{1}_{1:T})\pi^2_T(x^{2}_{1:T})\delta_{x^0_T}(x^0)=0, \label{special_cost:DP1}
\end{align}
where we used the assumption on the cost function, namely, $k_T(x_T,m(x^1_T),m(x^2_T))=0$. Hence,
\begin{align}
    &Q_{T}(\pi^1_T,\pi^2_T,x^0_T,\gamma_T)\ge Q_{T}(\pi^1_T,\pi^2_T,x^0_T,m)=0,
  \end{align} 
and therefore,
  \begin{align}
    &V_{T}(\pi^1_T,\pi^2_T,x^0_T) =\min_{\gamma_T \in\mathcal{B}_T}Q_{T}(\pi^1_T,\pi^2_T,x^0_T,\gamma_T)=0. \label{eq:value_spec}
\end{align} 
\emph{Induction hypothesis:} Assume the coordinator's value function $V_{t+1}(\pi^1_{t+1},\pi^2_{t+1},x^0_{t+1})=0$ for any realization $\pi^1_{t+1},\pi^2_{t+1},x^0_{t+1}$ at time $t+1$. 

At time $t$ we define the function $Q_t$ as follows:
  \begin{align}
  &Q_t(\pi^1_t,\pi^2_t,x^0_t,\gamma_t):=\sum_{x_{1:t}}\sum_{u_t}\pi^1_t(x^{1}_{1:t})\pi^2_t(x^{2}_{1:t})\delta_{x^0_t}(x^0)\times\notag\\
  &\gamma_t(x^{1}_{1:t};u^1_t)\gamma_t(x^{2}_{1:t};u^2_t) k_{t}(x_t,u_t)+\mathbb{E}[V_{t+1}(\delta_{x^0_{t+1}},\notag\\
  &\Pi^1_{t+1},\Pi^2_{t+1})|(\Pi^1_t,\Pi^2_t,X^0_t,\Gamma_t)=(\pi^1_t,\pi^2_t,x^0_t,\gamma_t)]
 \end{align}
Because of the induction hypothesis, the expectation of the value function at time $t+1$ is $0$ and we can simplify $Q_{t}$   as follows:
 \begin{align}
  Q_{t}(\pi^1_{t},\pi^2_{t},x^0_{t},&\gamma_{t}):=\sum_{x_{1:{t}}}\sum_{u_{t}}\delta_{x^0_{t}}(x^0)\pi^1_{t}(x^{1}_{1:{t}})\pi^2_t(x^{2}_{1:t})\notag\\
    &\times\gamma_t(x^{1}_{1:t};u^1_{t})\gamma_t(x^{2}_{1:t};u^2_{t})k_{t}(x_t,u_t).
 \end{align}
Using the same arguments as those used for $Q_T$, it follows that 
\begin{align}
    &Q_{t}(\pi^1_t,\pi^2_t,x^0_t,\gamma_t)\ge Q_{t}(\pi^1_t,\pi^2_t,x^0_t,m)=0,
  \end{align} 
and therefore, $V_t(\pi^1_t,\pi^2_{t},x^0_t)=0$.  Thus, the induction hypothesis is true for all times.

It is clear from the above argument that the optimal prescription for the coordinator in Problem P1b is $m$ at each time and for any realization of its information state. Similar arguments can be repeated for the coordinator in Problem P1c as well.

\label{proof_special_cost}
\section{Proof of Lemma \ref{theorem:sp_dyn}}\label{App:sp_dyn}
Because of the specialized dynamics, the coordinator's belief on each agent's private information at  time $t$  is given by $\alpha^{1:t}$ for Problem P1b and  by $\alpha$ for Problem P1c.

At these beliefs, the value functions  for the coordinators  in Problems P1b and P1c are as follows: 
\begin{align}
    &V^b_t(\alpha^{1:t},\alpha^{1:t}) = \min_{\gamma^b_t \in \mathcal{B}_t} Q^b_{{t}}(\alpha^{1:t},\alpha^{1:t},\gamma^b_{t}),\notag\\
    &V^c_t(\alpha,\alpha) = \min_{\gamma^c_t \in \mathcal{B}_t} Q^c_{{t}}(\alpha,\alpha,\gamma^c_{t})\label{eq:sDP1}
\end{align}
where the functions $Q^b_t$ and $Q^c_t$ are defined as
\begin{align}
    &Q^b_{t}(\alpha^{1:t},\alpha^{1:t},\gamma^b_{t}):=\sum_{x_{1:{t}}}\sum_{u_{t}}\alpha^{1:t}(x^{1}_{1:{t}})\alpha^{1:t}(x^{2}_{1:t})\gamma^b_t(x^{1}_{1:t};u^1_{t})\notag\\
    &\times\gamma^b_t(x^{2}_{1:t};u^2_{t})k_{t}(x_t,u_t)+\mathbb{E}[V_{t+1}(\alpha^{1:t+1},\alpha^{1:t+1})],\notag\\
    &Q^c_{t}(\alpha,\alpha,\gamma^c_{t}):=\sum_{x_{{t}}}\sum_{u_{t}}\alpha(x^{1}_{t})\alpha(x^{2}_{t})\gamma^c_t(x^{1}_{t};u^1_{t})\gamma^c_t(x^{2}_{t};u^2_{t})\notag\\
    &\times k_{t}(x_t,u_t)+\mathbb{E}[V_{t+1}(\alpha,\alpha)].
\end{align}
 Using a backward inductive argument, we can show that for any $\gamma^b_{t}$ there exists a $\gamma^c_t$ such that $Q^b_{t}$ and $Q^c_{t}$ defined above are the same (such a  $\gamma^c_t$ must satisfy equations of the form: $\sum_{x^i_{1:{t-1}}}\gamma^b_{t}(x^i_{1:t})\alpha^{1:t}(x^i_{1:t})=\alpha(x^i_t)\gamma^c_t(x^i_t)$). Similarly, we can show that for any $\gamma^c_{t}$ there exists a $\gamma^b_t$ such that $Q^b_{t}$ and $Q^c_{t}$ are the same (such a $\gamma^b_t$ can be defined as $\gamma^b_t(x^i_{1:t}):=\gamma^c_t(x^i_t)$). This relationship between the two $Q$-functions 
 implies the following equation for the corresponding value functions:
\begin{align}
    V^b_t(\alpha^{1:t},\alpha^{1:t})=V^c_t(\alpha,\alpha) \label{eq:appC_value}.
\end{align}
The optimal cost in each problem is the value function at time $t=1$ evaluated at the prior belief $\alpha$. Therefore,  \eqref{eq:appC_value} at $t=1$ implies that the two problems have the same optimal performance.
Consequently, an optimal symmetric strategy  in Problem P1c will achieve the  optimal performance in Problem P1b as well. 

\section{Proof of Lemma \ref{lemma:1a-1c}} \label{App:ac}
To prove equivalence between Problems P1a and P1c, we show that the coordinator's dynamic program are identical in both cases.

Towards this we show that the coordinator belief is the same for both problems in the following lemma.
\begin{lemma}\label{LEM:UPDATEspecdyn}
For $i=1,2$ and for each $t \geq 1$, if $x^0_t,u_t$ is the realization of shared state $X^0_t$ and action $U_t$. The beliefs at time $t+1$ in the two problems are the same and given as follows:
\begin{align}
    \pi^i_{t+1}=\hat{\eta}^i_t(x^0_t,u_t)
\end{align}
where $\hat{\eta}^i_t(x^0_t,u_t)$ is the probability distribution of $f_t(x^0_t,u_t,W^i_t)$.
\end{lemma}
\begin{proof}
See Appendix II.
\end{proof}
At time $t=1$, the beliefs are the same for both problems and the above lemma proves that the beliefs are same for all time $t$.

Recall that the private information are same for both the problems, therefore this leads to same prescription space $\mathcal{B}_t$.
The Dynamic program in Theorem \ref{thm:dp} when applied to problem P1a and P1c are as follows:
\begin{corollary}\label{thm:dpspecial_dyna}
Define $V^a_{T+1}(x^0_{t+1},\pi^1_{T+1},\pi^2_{T+1}) =0$ and $V^c_{T+1}(x^0_{t+1},\pi^1_{T+1},\pi^2_{T+1}) =0$ for every $x^0_{t+1},\pi^1_{T+1},\pi^2_{T+1}$. For $t \leq T$ and  for any realization $x^0_{t},\pi^1_{t},\pi^2_{t}$ of $X^0_{t},\Pi^1_{t},\Pi^2_{t}$, define
\begin{align}
    &V^a_t(x^0_{t},\pi^1_{t},\pi^2_{t}) = \min_{\gamma_t \in \mathcal{B}_t} \mathbb{E}[k_t(X_t, U_t) +V^a_{t+1}(\delta_{X^0_{t+1}}, \\ \notag 
    &\hat{\eta}^1_t(x^0_t,U_t),\hat{\eta}^2_t(x^0_t,U_t)) |(X^0_{t},\Pi^1_{t},\Pi^2_{t},\Gamma_t)=(x^0_{t},\pi^1_{t},\pi^2_{t},\gamma_t)]\label{eq:DP1}
\end{align}
\begin{align}
    &V^c_t(x^0_{t},\pi^1_{t},\pi^2_{t}) = \min_{\gamma_t \in \mathcal{B}_t} \mathbb{E}[k_t(X_t, U_t) +V^c_{t+1}(\delta_{X^0_{t+1}}, \\ \notag 
    &\hat{\eta}^1_t(x^0_t,U_t),\hat{\eta}^2_t(x^0_t,U_t)) |(X^0_{t},\Pi^1_{t},\Pi^2_{t},\Gamma_t)=(x^0_{t},\pi^1_{t},\pi^2_{t},\gamma_t)]\label{eq:DP2}
\end{align}
\end{corollary}
Using the above Corollary, we show that in both the problems value functions are the same for any $\pi \in \delta(x)$.
 \begin{lemma}
  For $t\le T$, let $V^a_{t}(\cdot)$ be the value function for the coordinator's dynamic program in problem P1a and $V^c_{t}(\cdot)$ be the value function for the coordinator's dynamic program in  problem P1c. We show that $V^a_{t}(\pi)=V^c_{t}(\pi)$ for any $\pi \in \delta(x)$.   
 \end{lemma}
 \begin{proof}
 By definition value functions for both problems at time $T+1$ are same (equal to $0$). At time $T$, the value functions evaluated at the same belief leads to $(X_T,U_T)$ has a pair having the same probability distribution. Therefore, the expectations are the same and the minimization is over identical prescription space proves that the value functions are the same in both the problems.
 So using the above argument inductively backwards in time. At any time $t$, the expectation over the cost function are same in both problems because $(X_t,U_t)$ has a pair have the same probability distribution when evaluated over the same belief. Using Lemma \ref{LEM:UPDATEspecdyn}, the beliefs at time $t+1$ in the two problems are the same and from the inductive step value functions evaluated at same belief are equal.
Hence the expectations are the same and minimization over identical prescription space proves that the value functions are same for both the problems.
 \end{proof}

The optimal cost in each problem is the expectation of value function at time $t=1$ evaluated at prior belief and initial shared state $X^0_1$. Since value functions are same for both the problems P1a and P1c, optimal costs are equal.

\label{App:ac}
\section{State dynamics influenced by an aggregate of agent actions rather than individual actions.} \label{App:at}

For problem P1d Lemma \ref{LEM:INDEPEN} doesn't hold true.
\begin{equation}\label{eq:indepen}
    \prob^{(g^1,g^2)}(x^1_{t},x^2_t|c_t)\neq\displaystyle\prod_{i=1}^{2} \prob^{g^i}(x^{i}_{t}|c_t), 
    \end{equation}
The proof of the same is given below:
\begin{proof}
\begin{align*}
    &\prob(x^1_{t},x^2_t|c_t)=  \prob(x^1_{t},x^2_t|x^0_{1:t},a_{1:t-1})  \\
    &=\frac{\sum_{x^{1,2}_{t-1}}\sum_{u_{t-1}}\prob(x^{1,2}_t,x^0_t,a_{t-1},x^{1,2}_{t-1},u_{t-1}|x^0_{1:t-1},a_{1:t-2})}{\prob(x^0_t,a_{t-1}|x^0_{1:t-1},a_{1:t-2})}\notag\\
    &=\frac{\sum_{x^{1,2}_{t-1}}\sum_{u_{t-1}}\prob(x^{1,2}_t,x^0_t,a_{t-1},x^{1,2}_{t-1},u_{t-1}|x^0_{1:t-1},a_{1:t-2})}{\sum_{\tilde{x}^{1,2}_{t-1}}\sum_{\tilde{u}_{t-1}}\prob(x^0_t,a_{t-1},\tilde{x}^{1,2}_{t-1},\tilde{u}_{t-1}|x^0_{1:t-1},a_{1:t-2})} 
\end{align*}
 The numerator of the above equation can be written as:
 \begin{align*}
&\sum_{x^{1,2}_{t-1}}\sum_{u_{t-1}}\prob(x^1_t|x^1_{t-1},x^0_{t-1},a_{t-1})\prob(x^2_t|x^2_{t-1},x^0_{t-1},a_{t-1})\notag\\
&\prob(x^0_t|x^0_{t-1},a_{t-1})\prob(a_{t-1}|u_{t-1})\prob(u^1_{t-1}|x^1_{t-1},c_t)\notag\\
&\prob(u^2_{t-1}|x^2_{t-1},c_t)\prob(x^1_{t-1}|x^0_{1:t-1},a_{1:t-2})\prob(x^2_{t-1}|x^0_{1:t-1},a_{1:t-2})
 \end{align*}
Similarly the denominator can be written as:
 \begin{align*}
&\sum_{\tilde{x}^{1,2}_{t-1}}\sum_{\tilde{u}_{t-1}}\prob(x^0_t|x^0_{t-1},a_{t-1})\prob(a_{t-1}|\tilde{u}_{t-1})\prob(\tilde{u}^1_{t-1}|\tilde{x}^1_{t-1},c_t)\notag\\
&\prob(\tilde{u}^2_{t-1}|\tilde{x}^2_{t-1},c_t)\prob(\tilde{x}^1_{t-1}|x^0_{1:t-1},a_{1:t-2})\prob(\tilde{x}^2_{t-1}|x^0_{1:t-1},a_{1:t-2})
 \end{align*}
The problem arises because of the term $\prob(a_{t-1}|u_{t-1})$  which does not factorize into a separate agent's action. Therefore, problem P1d Lemma \ref{LEM:INDEPEN} doesn't hold true.
\end{proof}
One example to prove Lemma \ref{LEM:INDEPEN} doesn't hold true for problem P1d is:
Consider a scenario in which there is no shared state, the action space is denoted as $\mathcal{U} =  \{0, 1\}$, and the state space is represented by $\mathcal{X} = \{0, 1\}$. Let $T=2$, and it is assumed that the local state of agent $1$ remains stationary over time. The dynamics of agent $2$ are defined by the transition probabilities $\prob(x^2_t=a|x^2_{t-1}=a)$ = 0.8, for $a \in  \{0, 1\}$. The initial states $X^1_1$ and $X^2_1$ are independent random variables with probability distribution Bernoulli (1/2). The aggregate action $a_t$ is the sum of individual agents actions $u^1_t $and $u^2_t$. Consider the  strategy pair $(g,g)$ for Problem P1d where $g_t$ (the strategy at time $t$)  is:
\begin{equation}
g_t(u^i_t=0|x^i_t)=
 \begin{cases}
    1, ~ \text{if}~  x^i_t=0\\
    0, ~ \text{if}~  x^i_t=1,
  \end{cases} \label{strategy_exp}
\end{equation}
To disprove Lemma 7, we first need to find the joint probability conditioned on the common information. Here in this problem the common information ($C_t$) is the aggregate action $A_t$. 
The joint probability at time $t=2$ is computed under the assumption that the common information at time $t=2$ is $a_1=u^1_1+u^2_1=1$. Specifically, we are interested in evaluating the probability of the event $x^1_2=0, x^2_2=0$ conditioned on $a_1=1$. Thus, the expression can be written as:

\begin{align*}
&\prob(x^1_2=0, x^2_2=0|c_2) = \prob(x^1_2=0, x^2_2=0|a_1=1)\notag\\
&=\sum_{x^{1,2}_1}\sum_{u^{1,2}_1}\prob(x^1_2=0, x^2_2=0,x^{1,2}_1,u^{1,2}_1|a_1=1)\notag\\
&=\frac{\sum_{x^{1,2}_1}\sum_{u^{1,2}_1}\prob(x^1_2=0, x^2_2=0,x^{1,2}_1,u^{1,2}_1,a_1=1)}{\sum_{\tilde{x}^{1,2}_1}\sum_{\tilde{u}^{1,2}_1}\prob(\tilde{x}^{1,2}_1,\tilde{u}^{1,2}_1,a_1=1)}
\end{align*}
Consider the numerator in the above equation,
\begin{align*}
\sum_{x^{1,2}_1}\sum_{u^{1,2}_1}&\prob(x^1_2=0|x^1_1)\prob(x^2_2=0|x^2_1)\prob(a_1=1|u^1_1,u^2_1)\notag\\
&\times\prob(u^1_1|x^1_1)\prob(u^2_1|x^2_1)\prob(x^1_1,x^2_1)
\end{align*}
and the denominator is,
\begin{align*}
\sum_{\tilde{x}^{1,2}_1}\sum_{\tilde{u}^{1,2}_1} \prob(a_1=1|\tilde{u}^1_1,\tilde{u}^2_1)\prob(\tilde{u}^1_1|\tilde{x}^1_1)\prob(\tilde{u}^2_1|\tilde{x}^2_1)\prob(\tilde{x}^1_1,\tilde{x}^2_1) 
\end{align*}
The joint probability is evaluated to,
\begin{align}
    \prob(x^1_2=0, x^2_2=0|a_1=1)=\frac{0.2*0.25}{0.5}=0.1
\end{align}

Now considering the marginal conditional probability,
\begin{align*}
   & \prob(x^1_2=0|a_1=1)=\sum_{x^{1,2}_1}\sum_{u^{1,2}_1}\prob(x^1_2=0, x^{1,2}_1,u^{1,2}_1|a_1=1)\notag\\
    &=\frac{\sum_{x^{1,2}_1}\sum_{u^{1,2}_1}\prob(x^1_2=0, x^{1,2}_1,u^{1,2}_1,a_1=1)}{\sum_{\tilde{x}^{1,2}_1}\sum_{\tilde{u}^{1,2}_1}\prob( \tilde{x}^{1,2}_1,\tilde{u}^{1,2}_1,a_1=1)}
\end{align*}
The numerator of the above equation can be written as,
\begin{align*}
\sum_{x^{1,2}_1}\sum_{u^{1,2}_1}\prob(x^1_2=0|x^1_1)&\prob(a_1=1|u^1_1,u^2_1)\prob(u^1_1|x^1_1)\prob(u^2_1|x^2_1)\notag\\
&\times\prob(x^1_1,x^2_1)
\end{align*}
and the denominator is,
\begin{align*}
 \sum_{\tilde{x}^{1,2}_1}\sum_{\tilde{u}^{1,2}_1}\prob(a_1=1|\tilde{u}^1_1,\tilde{u}^2_1)\prob(\tilde{u}^1_1|\tilde{x}^1_1)\prob(\tilde{u}^2_1|\tilde{x}^2_1)\prob(\tilde{x}^1_1,\tilde{x}^2_1) 
\end{align*}
the marginal conditional probability is evaluated to,
\begin{align}
\prob(x^1_2=0|a_1=1)=0.5
\end{align}
Similarly evaluating the marginal conditional probability for agent $2$,
\begin{align*}
&\sum_{x^{1,2}_1}\sum_{u^{1,2}_1}\prob(x^2_2=0|x^2_1)\prob(a_1=1|u^1_1,u^2_1)\prob(u^1_1|x^1_1)\prob(u^2_1|x^2_1)\notag\\
&\times\frac{\prob(x^1_1,x^2_1)}{\sum_{\tilde{x}^{1,2}_1}\sum_{\tilde{u}^{1,2}_1} \prob(a_1=1|\tilde{u}^1_1,\tilde{u}^2_1)\prob(\tilde{u}^1_1|\tilde{x}^1_1)\prob(\tilde{u}^2_1|\tilde{x}^2_1)\prob(\tilde{x}^1_1,\tilde{x}^2_1) }
\end{align*}
which evaluates to,
\begin{align}
\prob(x^2_2=0|a_1=1)=0.5
\end{align}
Hence proved that, $\prob(x^1_{2}=0,x^2_2=0|a_1)\neq\displaystyle\prod_{i=1}^{2} \prob(x^{i}_{2}=0|a_1).$

\begin{lemma} \label{lemma:belief_update}
In Problem P1d For any realization $x^0$ of the global state and any realizations $x^1$ and $x^2$ of the agents’ private information,
\begin{equation}
    \pi_{t}(x^0,x^1,x^2)=\delta_{X^0_t}{(x^0)}\pi^{1,2}_{t}(x^{1},x^{2})
\end{equation}
For any coordination strategy $d$, the coordinator's belief $\Pi_t$ in problem P1d evolves almost surely as
\begin{equation}
    \Pi^{1,2}_{t+1} = \eta_t(X^0_t,\Pi^{1,2}_t, \Gamma_t, A_t),
\end{equation}
where $\eta_t$ is a fixed transformation that does not depend on the coordination strategy.
\end{lemma}

\begin{proof}
The coordinator’s belief state can serve as a sufficient statistic for selecting prescriptions. Let $c_{t+1}:=(x^0_{t+1},a_{1:t})$ be the realization of common information $C_{t+1}$ and $\gamma_{1:t}$ be the realization of the prescription $\Gamma_{1:t}$. The coordinator belief in Problem P1d for all $x^i_{t+1} \in \mathcal{X}$ is given as:
   \begin{align*}
        &\pi^{1,2}_{t+1}(x^{1,2}_{t+1}) = \prob(x^{1,2}_{t+1}|c_{t+1}=(x^0_{t+1},a_{1:t}),\gamma_{1:t})\notag\\
        & =\frac{\prob(x^{1,2}_{t+1},a_t|x^0_{t},a_{1:t-1},\gamma_{1:t})}{\prob(a_t|x^0_{t},a_{1:t-1},\gamma_{1:t})}\notag\\
&=\frac{\sum_{x^{1,2}_{t}}\sum_{u_t}\prob(x^{1,2}_{t+1},a_t,x^{1,2}_t,u_t|x^0_{t},a_{1:t-1},\gamma_{1:t})}{\sum_{\tilde{x}^{1,2}_{t}}\sum_{\tilde{u}_t}\prob(a_t,\tilde{x}^{1,2}_t,\tilde{u}_t|x^0_{t},a_{1:t-1},\gamma_{1:t})} \label{bel_nw}
   \end{align*} 
   The numerator of the above equation can be written as:
\begin{align*}
&\sum_{x^{1,2}_{t}}\sum_{u_t}\prob(x^{1}_{t+1}|x^{1}_t,x^0_{t},a_t)\prob(x^{2}_{t+1}|x^{1}_t,x^0_{t},a_t)\prob(a_t|u_t)\notag\\
&\times\gamma(x^1_t;u^1_t)\gamma(x^2_t;u^2_t)\pi^{1,2}_t(x_t)
\end{align*}
Similarly, the denominator can be written as:
\begin{align*}
&\sum_{x^{1,2}_{t}}\sum_{u_t}\prob(a_t|u_t)\gamma(x^1_t;u^1_t)\gamma(x^2_t;u^2_t)\pi^{1,2}_t(x_t)
\end{align*}
Thus $\pi_{t+1}(\cdot)$ is determined by $x^0_t,\pi^{1,2}_t, \gamma_t, a_t$. We denote the update rule described above with $\eta_t$.
\end{proof}

\label{App:at}







\end{document}